\documentclass[12pt]{amsart}
\usepackage{amsfonts,amsthm,amsmath,amscd,amssymb}
\usepackage[T1]{fontenc}
\usepackage[mathscr]{eucal}
\usepackage{indentfirst}
\usepackage{graphicx}
\usepackage{graphics}
\usepackage{xcolor}
\usepackage{pict2e}
\usepackage{epic}
\usepackage{hyperref}
\numberwithin{equation}{section}
\usepackage[margin=2.9cm]{geometry}
\usepackage{epstopdf} 
\usepackage{cite}
\usepackage{parskip}
\usepackage{array}
\usepackage{booktabs}
\usepackage{enumerate}
\usepackage{adjustbox}
\usepackage{tikz}

 \makeatletter
\newcommand{\addresseshere}{%
  \enddoc@text\let\enddoc@text\relax
}
\makeatother

\theoremstyle{plain}
\newtheorem{thm}{Theorem}[section]
\newtheorem{lem}[thm]{Lemma}

\newtheorem{prop}[thm]{Proposition}

 \theoremstyle{definition}

\newtheorem{rem}[thm]{Remark}
\newtheorem{notn}[thm]{Notation}

\newcommand{\mb}[1]{\mathbb{#1}}

\newcommand{\mf}[1]{\mathfrak{#1}}
\newcommand{\mr}[1]{\mathrm{#1}}
\newcommand{\eps}{\varepsilon}

\newcommand{\pcoor}[1]{%
  \begingroup\lccode`~=`: \lowercase{\endgroup
  \edef~}{\mathbin{\mathchar\the\mathcode`:}\nobreak}%
  [% opening symbol
  \begingroup
  \mathcode`:=\string"8000
  #1%
  \endgroup 
  ]% closing symbol
}

\begin{document}
%%%%%%%%%%%%%%%%%%%%%%%%%%%%%%%%%%%%%%%%
\title{All lines on a smooth cubic surface in terms of three skew lines}

\author{Stephen McKean}
\address{Department of Mathematics \\ Duke University \\ Durham \\ NC} 
\email{mckean@math.duke.edu}

\author{Daniel Minahan}
\address{School of Mathematics \\ Georgia Institute of Technology \\ Atlanta \\ GA} 
\email{dminahan6@gatech.edu}

\author{Tianyi Zhang}
\address{School of Mathematics \\ Georgia Institute of Technology \\ Atlanta \\ GA} 
\email{kafuka@gatech.edu}

%%%%%%%%%%%%%%%%%%%%%%%%%%%%%%%%%%%%%%%%

\begin{abstract}
Jordan showed that the incidence variety of a smooth cubic surface containing 27 lines has solvable Galois group over the incidence variety of a smooth cubic surface containing 3 skew lines. As noted by Harris, it follows that for any smooth cubic surface, there exist formulas for all 27 lines in terms of any 3 skew lines. In response to a question of Farb, we compute these formulas explicitly. We also discuss how these formulas relate to Schl\"afli's count of lines on real smooth cubic surfaces.
\end{abstract}

\maketitle

\section{Introduction}
Given a complex smooth cubic surface $S$ containing three skew lines, we compute the equations of all 27 lines on $S$. We then apply these equations to study lines on real smooth cubic surfaces. Schl\"afli showed that real smooth cubic surfaces contain 3, 7, 15, or 27 lines~\cite{Sch58}. Moreover, a real smooth cubic surface contains three skew lines if and only if the cubic surface contains an \textit{elliptic line} (as defined by Segre \cite{Seg42}). Given a real smooth cubic surface $S$ that contains an elliptic line, we give a cubic polynomial $g(t)\in\mb{R}[t]$ (Proposition~\ref{prop:c_i}) and a pair of complex numbers $s_1,s_2$ (Notation~\ref{notn:s_i}) that determine the number of real lines contained in $S$.

\begin{thm}\label{thm:main theorem}
Let $S$ be a real smooth cubic surface that contains an elliptic line.
\begin{enumerate}[(a)]
\item $S$ contains exactly 7 real lines if and only if $g(t)$ has only one real root and $s_1,s_2$ are not real.
\item $S$ contains exactly 15 real lines if and only if (i) all roots of $g(t)$ are real and $s_1,s_2$ are not real, or (ii) $g(t)$ has only one real root and $s_1,s_2$ are real.
\item $S$ contains 27 real lines if and only if all roots of $g(t)$ and $s_1,s_2$ are real.
\end{enumerate}
\end{thm}

\begin{rem}
Let $k\subseteq\mb{C}$ be a field. Let $S$ be a smooth cubic surface over $k$ with three skew $k$-rational lines. As pointed out by the referee, the equations in this paper allow one to characterize the number of $k$-rational lines on $S$. See \cite{McK21} for a similar application of these equations to the study of lines on cubic surfaces over $\mb{Q}$.
\end{rem}

In algebraic geometry, enumerative problems can often be rephrased in terms of covering spaces of incidence varieties. By studying the monodromy of these covers, one can speak of the Galois group of an enumerative problem. These Galois groups can provide additional insight into the enumerative problems at hand. For example, Jordan showed that the Galois group of 27 lines on a smooth cubic surface is the odd orthogonal group $O^-_6(\mb{Z}/2\mb{Z})\leq S_{27}$~\cite{Jor57} (see also \cite[pp. 715-718]{Har79}). Since $O^-_6(\mb{Z}/2\mb{Z})$ is not a solvable group, there is no equation in radicals for the 27 lines on a given smooth cubic surface. However, given a smooth cubic surface and a particular arrangement of lines contained therein, we obtain a new Galois group $G\leq O^-_6(\mb{Z}/2\mb{Z})$ that may be solvable.

Let $S$ be a smooth cubic surface over an algebraically closed field of characteristic 0. Let $\mb{P}^{19}$ be the projective space parametrizing cubic surfaces in $\mb{P}^3$, and let $\mb{G}(1,3)$ be the Grassmannian of lines in $\mb{P}^3$. Consider the incidence varieties
\begin{align*}
\Phi_{27}&=\{(S,L_1,...,L_{27})\in\mb{P}^{19}\times\mb{G}(1,3)^{27}:L_i\subseteq S\text{ for all }i\},\\
\Phi_{3,\text{skew}}&=\{(S,L_1,L_2,L_3)\in\mb{P}^{19}\times\mb{G}(1,3)^3:\\
&\qquad L_i\subseteq S\text{ for all }i\text{ and }L_i\cap L_j=\emptyset\text{ for all }i\neq j\}.
\end{align*}

Jordan showed that the covering $\Phi_{27}\to\Phi_{3,\text{skew}}$ has Galois group of order 12, which is thus solvable~\cite{Jor57}. In fact, Harris noted that this Galois group is dihedral \cite[p. 718]{Har79}. Because the Galois group of $\Phi_{27}\to\Phi_{3,\text{skew}}$ is solvable, Harris remarked there exists a formula in radicals for all 27 lines on a smooth cubic surface in terms of the cubic surface and any three skew lines that it contains \cite[pp. 718--719]{Har79}. At the {\it Roots of Topology} workshop at the University of Chicago in 2018, Benson Farb asked if these formulas could be written out explicitly. The bulk of this paper is devoted to giving explicit equations for all lines on a smooth cubic surface in terms of any three skew lines on the same surface. We then use these equations to prove Theorem~\ref{thm:main theorem} in Section~\ref{sec:real case}.

\subsection{Outline}
The layout of the paper is as follows. In Section~\ref{sec:notation}, we introduce notation and conventions for the paper. In Sections~\ref{sec:step1} through~\ref{sec:step5}, we assume that we are given a smooth cubic surface $S$ containing the skew lines $E_1=\mb{V}(x_0,x_1)$, $E_2=\mb{V}(x_2,x_3)$, and $E_3=\mb{V}(x_0-x_2,x_1-x_3)$ and solve for the remaining 24 lines. In Section~\ref{sec:general}, we solve the general case using a projective change of coordinates. We discuss how the formulas obtained in this paper relate to Schl\"afli's enumeration of real lines on smooth cubic surfaces over $\mb{R}$ in Section~\ref{sec:real case}. In Appendix~\ref{sec:graphics}, we include visualizations of real cubic surfaces with 27, 15, and 7 lines. We are greatly indebted to Steve Trettel for preparing these graphics. In Appendix~\ref{sec:table}, we list the equations of all 27 lines on a smooth cubic surface containing $E_1,E_2,E_3$.

Given three skew lines on a smooth cubic surface, there are various ways to geometrically recover the remaining 24 lines. Harris describes one such method~\cite[pp. 718-719]{Har79}, which we utilize for most of our approach. However, we occasionally apply a different geometric method than Harris's when this simplifies the resulting computations. In Section~\ref{sec:step1}, we consider the quadric surface $Q$ defined by the skew lines $E_1,E_2,E_3$. These lines are contained in one ruling of $Q$, and the other ruling intersects $S$ in precisely three skew lines $C_4,C_5,C_6$. In Section~\ref{sec:step2}, we intersect $S$ with the planes spanned by $E_i$ and $C_j$. Each of these intersections consists of three lines by B\'ezout's Theorem; these lines are $E_i,C_j$, and $L_{i,j}$. For the next step, Harris suggests solving a quadratic equation defined by Pl\"ucker relations. This proved to be difficult in the generality needed for this paper, so we use a different approach in Section~\ref{sec:step3}. In particular, the four lines $E_1,E_2,L_{3,4},L_{3,5}$ are skew, so there are exactly two lines, called $C_3$ and $L_{1,2}$, meeting all four of these skew lines. Following Eisenbud and Harris~\cite[3.4.1]{EH16}, we let $Q'$ be the quadric surface defined by $E_1,E_2,L_{3,4}$. By B\'ezout's Theorem, $Q'\cap L_{3,5}$ consists of two points. Each of these points is contained in a line in the ruling that does not contain $E_1,E_2,L_{3,4}$; these two lines are $C_3$ and $L_{1,2}$. In Section~\ref{sec:step4}, we solve for four more lines. Here, the general technique is to repeat the process of Section~\ref{sec:step2}, using projective changes of coordinates as needed. While Harris suggests computing the remaining ten lines in this manner, the method becomes complicated for the lines $E_4,E_5,E_6,L_{4,5},L_{4,6}$, and $L_{5,6}$. In Section~\ref{sec:step5} we solve for these final six lines using the same process as in Section~\ref{sec:step3}.

\subsection{Related work}
Pannizut, Sert\"oz, and Sturmfels~\cite{PSS19} also give explicit equations for certain lines on smooth cubic surfaces. Let $S$ be a smooth cubic surface whose defining polynomial $f=\sum_{i+j+k+l=3}\alpha_{i,j,k,l}x_0^ix_1^jx_2^kx_3^l$ has full support (that is, $\alpha_{i,j,k,l}\neq 0$ for all $i+j+k+l=3$). Pick 6 skew lines contained in $S$ and label them $E_1,...,E_6$. Then there exists a unique blow-down $\pi:S\to\mb{P}^2$ that sends $E_1,...,E_6$ to distinct points with $\pi(E_1)=\pcoor{1:0:0}$, $\pi(E_2)=\pcoor{0:1:0}$, $\pi(E_3)=\pcoor{0:0:1}$, and $\pi(E_4)=\pcoor{1:1:1}$. The authors give local charts $\{U\}$ on $S$ and formulas for the quadratic maps $\{\pi|_U:U\to\mb{P}^2\}$~\cite[Theorem 4.2]{PSS19}. All lines on $S$ can be recovered by $\pi^{-1}$, so this result gives equations for all lines on a smooth cubic surface (whose defining polynomial has full support) in terms of 6 skew lines.

\subsection{Acknowledgements}
We thank Benson Farb for asking this paper's motivating question. We also thank Matt Baker, Dan Margalit, Joe Rabinoff, Bernd Sturmfels, and Jesse Wolfson for helpful suggestions and support. We thank the anonymous referee for their detailed comments and suggestions that greatly improved the clarity and accuracy of the paper. Finally, we are especially grateful to Steve Trettel for the included graphics.

\section{Notation and conventions}\label{sec:notation}
Throughout this paper, we will be working in $\mb{P}^3:=\mb{P}^3_\mb{C}=\operatorname{Proj}(\mb{C}[x_0,x_1,x_2,x_3])$.

\subsection{Lines on cubic surfaces}
Following~\cite{Har79}, we denote the 27 lines on a smooth cubic surface $S$ by $E_i,C_j$ for $1\leq i,j\leq 6$ and $L_{i,j}$ for $i\neq j$ and $1\leq i,j\leq 6$. As Harris describes~\cite[p. 717]{Har79}, there are 72 different sets of six disjoint lines on $S$:
\begin{align*}
&\{E_i\}_{i=1}^6,\\
&\{E_i,E_j,E_k,L_{m,n}\}_{m,n\neq i,j,k},\\
&\{E_i,C_i,L_{j,k}\}_{k\neq i},\\
&\{C_i,C_j,C_k,L_{m,n}\}_{m,n\neq i,j,k},\\
&\{C_i\}_{i=1}^6.
\end{align*}

\subsection{Cubic surface}\label{sec:surface}
For the rest of the paper, let $S=\mb{V}(f)$ be a smooth cubic surface containing the skew lines $E_1=\mb{V}(x_0,x_1),E_2=\mb{V}(x_2,x_3)$, and $E_3=\mb{V}(x_0-x_2,x_1-x_3)$, where
\[f(x_0,x_1,x_2,x_3)=\sum_{i+j+k+l=3}\alpha_{i,j,k,l}x_0^ix_1^jx_2^kx_3^l.\]
Since $S$ contains $E_1,E_2,E_3$, it follows that $f(0,0,x_2,x_3)=f(x_0,x_1,0,0)=f(x_0,x_1,x_0,x_1)=0$. Evaluating $f(1,0,0,0)$, $f(0,1,0,0)$, $f(0,0,1,0)$, $f(0,0,0,1)$, $f(1,1,0,0)$, $f(0,0,1,1)$, $f(1,0,1,0)$, $f(0,1,0,1)$, $f(1,1,1,1)$, and $f(1,-1,1,-1)$ induces the following relations:
\begin{align}\label{relations}
    \alpha_{3,0,0,0}=\alpha_{0,3,0,0}=\alpha_{0,0,3,0}=\alpha_{0,0,0,3}&=0,\\
    \alpha_{2,1,0,0}=\alpha_{1,2,0,0}=\alpha_{0,0,2,1}=\alpha_{0,0,1,2}&=0,\nonumber\\
    \alpha_{0,2,0,1}+\alpha_{0,1,0,2}=\alpha_{2,0,1,0}+\alpha_{1,0,2,0}&=0,\nonumber\\
    \alpha_{0,2,1,0}+\alpha_{1,0,0,2}+\alpha_{1,1,0,1}+\alpha_{0,1,1,1}&=0,\nonumber\\
    \alpha_{0,1,2,0}+\alpha_{2,0,0,1}+\alpha_{1,0,1,1}+\alpha_{1,1,1,0}&=0.\nonumber
\end{align}

\subsection{Projective change of coordinates}
An invertible matrix $A\in\mr{GL}_4(\mb{C})$ gives a projective change of coordinates by $\pcoor{a_0:a_1:a_2:a_3}\mapsto\pcoor{b_0:b_1:b_2:b_3}$, where $(b_0,b_1,b_2,b_3)^T=A(a_0,a_1,a_2,a_3)^T$. By slight abuse of notation, we also denote this projective change of coordinates by $A:\mb{P}^3\to\mb{P}^3$. Given a variety $X=\mb{V}(g_1,...,g_n)$, the change of coordinates $A$ takes $X$ to $AX=\mb{V}(g_1\circ A^{-1},...,g_n\circ A^{-1})$. We also note that if $\ell=\sum a_ix_i$ is a linear function and $\ell\circ A^{-1}=\sum b_ix_i$, then $(A^{-1})^T(a_0,a_1,a_2,a_3)^T=(b_0,b_1,b_2,b_3)^T$.

\section{Three lines from a biruled quadric surface}\label{sec:step1}
The three skew lines $E_1,E_2,E_3$ define the quadric surface $Q=\mb{V}(x_0x_3-x_1x_2)$. Moreover, $Q$ contains the rulings $M_s=\{\pcoor{s:as:1:a}\in\mb{P}^3\}$ and $N_t=\{\pcoor{t:1:bt:b}\in\mb{P}^3\}$, with $M_\infty=\{\pcoor{1:a:0:0}\}$ and $N_\infty=\{\pcoor{1:0:b:0}\}$. Note that $M_0=E_1$, $M_\infty=E_2$, and $M_1=E_3$.

\begin{prop}\label{prop:c_i}
Let $t_4,t_5,t_6$ be the roots of
\begin{align*}
g(t)&=(\alpha_{2,0,1,0})t^3+(\alpha_{2,0,0,1}+\alpha_{1,1,1,0})t^2+(\alpha_{0,2,1,0}+\alpha_{1,1,0,1})t+\alpha_{0,2,0,1}\\
&=-((\alpha_{1,0,2,0})t^3+(\alpha_{0,1,2,0}+\alpha_{1,0,1,1})t^2+(\alpha_{1,0,0,2}+\alpha_{0,1,1,1})t+\alpha_{0,1,0,2}).
\end{align*}
Then $C_4=\mb{V}(x_0-t_4x_1,x_2-t_4x_3)$, $C_5=\mb{V}(x_0-t_5x_1,x_2-t_5x_3)$, and $C_6=\mb{V}(x_0-t_6x_1,x_2-t_6x_3)$.
\end{prop}
\begin{proof}
The lines $C_4,C_5,C_6$ are contained in both the cubic surface $S$ and the ruling $N_t$. A line $\{\pcoor{t:1:bt:b}:b\in\mb{C}\}$ is contained in $S$ if and only if $f(t,1,bt,b)=0$ for all $b$. Expanding this out and simplifying via the relations given in Equation~\ref{relations}, we have
\begin{align*}
f(t,1,bt,b)=(b-b^2)g(t),
\end{align*}
which vanishes for all $b\in\mb{C}$ if and only if $g(t)=0$. The roots $t_4,t_5,t_6$ of $g(t)$ will correspond to $C_4,C_5,C_6$. In particular, $N_{t_i}=\{\pcoor{t_i:1:bt_i:b}\}=\mb{V}(x_0-t_ix_1,x_2-t_ix_3)$ is a line contained in $S$. Since $N_{t_4},N_{t_5},N_{t_6}$ lie on the same ruling of $Q$, we may (without loss of generality) call them $C_4,C_5,C_6$, respectively. We also note that $t_i\neq t_j$ for $i\neq j$, or else we would have $C_i=C_j$, contradicting the overall count of 27 lines on a smooth cubic surface.
\end{proof}

\section{Nine residual lines}\label{sec:step2}
Next, we consider the planes $H_{i,j}$ spanned by $E_i$ and $C_j$ for $1\leq i\leq 3$ and $4\leq j\leq 6$. Intersecting each $H_{i,j}$ with $S$ will give a new line $L_{i,j}$ contained in $S$. In particular, since $E_i,C_j\subset S$, B\'ezout's Theorem implies that $S\cap H_{i,j}$ consists of $E_i,C_j$, and a third line.

\begin{prop}
We have the equations $L_{1,i}=\mb{V}(x_0-t_ix_1,\ell_{1,i})$, $L_{2,i}=\mb{V}(x_2-t_ix_3,\ell_{2,i})$, and $L_{3,i}=\mb{V}((x_0-x_2)-t_i(x_1-x_3),\ell_{3,i})$, where
\begin{align*}
\ell_{1,i}&=(t_i^2\alpha_{2,0,1,0}+t_i\alpha_{1,1,1,0}+\alpha_{0,2,1,0})x_1\\
&+(t_i\alpha_{1,0,2,0}+\alpha_{0,1,2,0})x_2\\
&+(t_i^2\alpha_{1,0,2,0}+t_i(\alpha_{0,1,2,0}+\alpha_{1,0,1,1})+\alpha_{0,1,1,1})x_3,\\
\ell_{2,i}&=(t_i\alpha_{2,0,1,0}+\alpha_{2,0,0,1})x_0\\
   &+(t_i^2\alpha_{2,0,1,0}+t_i(\alpha_{2,0,0,1}+\alpha_{1,1,1,0})+\alpha_{1,1,0,1})x_1\\
   &+(t_i^2\alpha_{1,0,2,0}+t_i\alpha_{1,0,1,1}+\alpha_{1,0,0,2})x_3,\\
\ell_{3,i}&=(t_i^2\alpha_{2,0,1,0}+t_i\alpha_{1,1,1,0}+\alpha_{0,2,1,0})x_1\\
    &+(t_i\alpha_{2,0,1,0}+\alpha_{0,1,2,0}+\alpha_{1,1,1,0})x_2\\
    &+(t_i\alpha_{2,0,0,1}-\alpha_{1,0,0,2})x_3.
\end{align*}
\end{prop}
\begin{proof}
Note that $H_{1,i}=\mb{V}(x_0-t_ix_1)$, $H_{2,i}=\mb{V}(x_2-t_ix_3)$, and $H_{3,i}=\mb{V}((x_0-x_2)-t_i(x_1-x_3))$. Since $S\cap H_{i,j}$ consists of three lines, it is given by the vanishing of a product of three linear homogeneous polynomials. Two of these factors will be given by $E_i$ and $C_j$, and the third will define $L_{i,j}$. The intersection $S\cap H_{1,i}$ is given by the vanishing of $f(t_ix_1,x_1,x_2,x_3)$ by substituting $x_0=t_ix_1$. The linear factors corresponding to $E_1$ and $C_i$ are $x_1$ and $x_2-t_ix_3$, respectively. By simplifying (using the relations from Equation~\ref{relations} when necessary), one can check that $f(t_ix_1,x_1,x_2,x_3)=x_1(x_2-t_ix_3)\ell_{1,i}$. It follows that $L_{1,i}$ is given by the vanishing of $x_0-t_ix_1$ and $\ell_{1,i}$.

Similarly, the intersection $S\cap H_{2,i}$ is given by the vanishing of $f(x_0,x_1,t_ix_3,x_3)=x_3(x_0-t_ix_1)\ell_{2,i}$, again using the given relations to simplify when necessary. Thus $L_{2,i}=\mb{V}(x_2-t_ix_3,\ell_{2,i})$. The intersection $S\cap H_{3,i}$ is given by the vanishing of $f(x_2+t_i(x_1-x_3),x_1,x_2,x_3)=(x_1-x_3)(x_2-t_ix_3)\ell_{3,i}$, again simplifying with the given relations. Thus $L_{3,i}=\mb{V}((x_0-x_2)-t_i(x_1-x_3),\ell_{3,i})$.
\end{proof}

\section{Two more lines from a quadric surface}\label{sec:step3}
To solve for the lines $C_3$ and $L_{1,2}$, we need to find the two lines that meet the four skew lines $E_1,E_2,L_{3,4},L_{3,5}$. We first give a projective change of coordinates $A$ such that $AE_1=E_1$, $AE_2=E_2$, and $AL_{3,4}=E_3$. We then intersect $AL_{3,5}$ with the quadric surface $Q=\mb{V}(x_0x_3-x_1x_2)$ defined by $E_1,E_2,E_3$. The intersection $Q\cap AL_{3,5}$ will consist of two points, which gives two lines in the ruling $N_t=\{\pcoor{t:1:bt:b}\}$, namely $AC_3$ and $AL_{1,2}$. We then obtain $C_3$ and $L_{1,2}$ by applying the projective change of coordinates $A^{-1}$.

\begin{notn}\label{notn:c_i and d_i}
Let
\begin{align*}
c_1&=t_4^2\alpha_{2,0,1,0}+t_4\alpha_{1,1,1,0}+\alpha_{0,2,1,0},\\
c_2&=t_4\alpha_{2,0,1,0}+\alpha_{0,1,2,0}+\alpha_{1,1,1,0},\\
c_3&=t_4\alpha_{2,0,0,1}-\alpha_{1,0,0,2},
\end{align*}
so that $\ell_{3,4}=c_1x_1+c_2x_2+c_3x_3$. Similarly, let
\begin{align*}
d_1&=t_5^2\alpha_{2,0,1,0}+t_5\alpha_{1,1,1,0}+\alpha_{0,2,1,0},\\ d_2&=t_5\alpha_{2,0,1,0}+\alpha_{0,1,2,0}+\alpha_{1,1,1,0},\\ d_3&=t_5\alpha_{2,0,0,1}-\alpha_{1,0,0,2},
\end{align*} 
so that $\ell_{3,5}=d_1x_1+d_2x_2+d_3x_3$.
\end{notn}

\begin{prop}
We have that $d_1\neq 0$.
\end{prop}
\begin{proof}
Suppose $d_1=0$. Then $L_{3,5}=\mb{V}((x_0-x_2)-t_5(x_1-x_3),d_2x_2+d_3x_3)$ contains the point $\pcoor{t_5:1:0:0}$, which is also contained in $E_2=\mb{V}(x_2,x_3)$. However, these lines are necessarily skew, so we obtain a contradiction. Thus $d_1\neq 0$.
\end{proof}

Consider the projective change of coordinates given by
\[A^T=\begin{pmatrix}1&0&0&0\\ -t_4&c_1&0&0\\ 0&0&1&-c_2\\ 0&0&-t_4&-c_3 \end{pmatrix}.\]
Note that $AE_1=E_1$, $AE_2=E_2$, and $AL_{3,4}=E_3$. Any projective change of coordinates in $\mb{P}^3$ is determined by its image on three skew lines. Moreover, $A$ takes the skew lines $E_1,E_2,L_{3,4}$ to the skew lines $E_1,E_2,E_3$. It follows that $A$ is non-singular, so $\det{A}=-c_1(c_3+t_4c_2)\neq 0$. Thus
\[(A^{-1})^T=\begin{pmatrix}1&0&0&0\\ \tfrac{t_4}{c_1}&\tfrac{1}{c_1}&0&0\\ 0&0&\tfrac{c_3}{c_3+c_2t_4}&-\tfrac{c_2}{c_3+c_2t_4}\\ 0&0&-\tfrac{t_4}{c_3+c_2t_4}&-\tfrac{1}{c_3+c_2t_4}\end{pmatrix}\]
is non-singular with $A^{-1}E_1=E_1$, $A^{-1}E_2=E_2$, and $A^{-1}E_3=L_{3,4}$.

\begin{notn}\label{notn:u_i and v_i}
Let
\begin{align*}
u_1&=\frac{t_4-t_5}{c_1},\\
u_2&=-\frac{c_3+c_2t_5}{c_3+c_2t_4},\\
u_3&=\frac{t_4-t_5}{c_3+c_2t_4},\\
v_2&=\frac{c_1}{d_1}\cdot\frac{d_2c_3-d_3c_2}{c_3+c_2t_4},\\
v_3&=-\frac{c_1}{d_1}\cdot\frac{d_2t_4+d_3}{c_3+c_2t_4},
\end{align*}
so that $AL_{3,5}=\mb{V}(x_0+u_1x_1+u_2x_2+u_3x_3,x_1+v_2x_2+v_3x_3)$.
\end{notn}

Recall that $E_1,E_2,E_3$ are contained in the ruling $M_s=\{[s:as:1:a]\}$ of $Q=\mb{V}(x_0x_3-x_1x_2)$. We will intersect $AL_{3,5}$ with $Q$ to obtain two lines in the ruling $N_t=\{\pcoor{t:1:bt:b}\}$. Substituting $x_0=-u_1x_1-u_2x_2-u_3x_3$ and $x_1=-v_2x_2-v_3x_3$ in the defining equation for $Q$, we find that
\[Q\cap AL_{3,5}=\mathbb{V}(v_2x_2^2+(u_1v_2-u_2+v_3)x_2x_3+(u_1v_3-u_3)x_3^2).\]

The points of $Q\cap AL_{3,5}$ are determined by the ratio $\frac{x_2}{x_3}$, so it suffices to solve the quadratic equation
\begin{align}\label{eq:quadratic}
v_2(\tfrac{x_2}{x_3})^2+(u_1v_2-u_2+v_3)\tfrac{x_2}{x_3}+(u_1v_3-u_3)=0.
\end{align}

By B\'ezout's Theorem, $Q\cap AL_{3,5}$ consists of two points. There are thus two distinct solutions to Equation~\ref{eq:quadratic}. In particular, $(u_1v_2-u_2+v_3)^2\neq 4v_2(u_1v_3-u_3)$. If $v_2\neq 0$, then we can use the quadratic formula to solve this equation.

\begin{rem}\label{rem:distinct}
If $Q\cap AL_{3,5}$ consisted of fewer than two points, then there would be at most one line meeting $E_1,E_2,E_3,L_{3,5}$, which would contradict the overall count of 27 distinct lines on $S$.
\end{rem}

\begin{notn}\label{notn:s_i}
Let
\[s_1=\frac{-(u_1v_2-u_2+v_3)+\sqrt{(u_1v_2-u_2+v_3)^2-4v_2(u_1v_3-u_3)}}{2v_2}\]
and
\[s_2=\frac{-(u_1v_2-u_2+v_3)-\sqrt{(u_1v_2-u_2+v_3)^2-4v_2(u_1v_3-u_3)}}{2v_2}\]
be the solutions of $v_2(\tfrac{x_2}{x_3})^2+(u_1v_2-u_2+v_3)\tfrac{x_2}{x_3}+(u_1v_3-u_3)=0$. If $(u_1v_2-u_2+v_3)^2-4v_2(u_1v_3-u_3)=re^{i\theta}$ with $r\geq 0$ and $0\leq\theta<2\pi$ is a complex number, then we denote $\sqrt{re^{i\theta}}=\sqrt{r}e^{i\theta/2}$ and $-\sqrt{re^{i\theta}}=-\sqrt{r}e^{i\theta/2}$.
\end{notn}

\begin{prop}\label{prop:c_3 and l_12}
We have the equations $C_3=\mathbb{V}(x_0+(-s_1c_1-t_4)x_1,(1+s_1c_2)x_2+(s_1c_3-t_4)x_3)$ and $L_{1,2}=\mathbb{V}(x_0+(-s_2c_1-t_4)x_1,(1+s_2c_2)x_2+(s_2c_3-t_4)x_3)$.
\end{prop}
\begin{proof}
Note that a line $\{\pcoor{t:1:bt:b}:b\in\mb{C}\}$ of the ruling $N_t$ is determined by the ratio $\frac{x_2}{x_3}=\frac{bt}{b}=t$. That is, the line $N_{s_i}=\mb{V}(x_0-s_i x_1,x_2-s_i x_3)$ contains the point of $Q\cap AL_{3,5}$ corresponding to $\frac{x_2}{x_3}=s_i$. Without loss of generality, we may denote $AC_3=\mb{V}(x_0-s_1x_1,x_2-s_1x_3)$ and $AL_{1,2}=\mb{V}(x_0-s_2x_1,x_2-s_2x_3)$. The proof is then completed by applying $A^{-1}AC_3=C_3$ and $A^{-1}AL_{1,2}=L_{1,2}$.
\end{proof}

\begin{rem}\label{rem:s_2=infty}
If $v_2=0$, then Equation~\ref{eq:quadratic} only has one root, say $s_1$. The other solution to this equation comes from $\frac{x_3}{x_2}=0$, which corresponds to $s_2=\infty$. Since $N_\infty=\mb{V}(x_1,x_3)$, we have $L_{1,2}=\mb{V}(c_1x_1,-c_2x_2-c_3x_3)=\mb{V}(x_1,c_2x_2+c_3x_3)$. This agrees with the formula $L_{1,2}=\mathbb{V}(x_0+(-s_2c_1-t_4)x_1,(1+s_2c_2)x_2+(s_2c_3-t_4)x_3)$ by dividing all terms by $s_2=\infty$.
\end{rem}

\section{Four lines as residual lines}\label{sec:step4}
Given our original three skew lines, along with the other fourteen lines that we have found, the remaining ten lines are residually determined. That is, given two lines $\Lambda_1,\Lambda_2$ in $S$, the intersection of $S$ with the plane $H$ containing $\Lambda_1$ and $\Lambda_2$ is a third line contained in $S$. The intersection $S\cap H$ is given by the vanishing of the product of three linear homogeneous polynomials; two of these factors correspond to $\Lambda_1$ and $\Lambda_2$, and the third factor corresponds to the desired line. We will frequently use projective changes of coordinates to simplify these computations. However, we only use this approach to find four of the remaining ten lines. Finding the lines $E_4,E_5,E_6,L_{4,5},L_{4,6}$, and $L_{5,6}$ proved to be difficult, so we give a different approach in Section~\ref{sec:step5}. We will use the fact that $E_i$ is residual to $C_j$ and $L_{i,j}$ if $i\neq j$ \cite[p. 719]{Har79}.

\subsection{$C_2$ and $L_{1,3}$}\label{sec:c2 and l13}
We will solve for $L_{1,3}$ and $C_2$ by applying the following discussion for $L=C_3$ and $L=L_{1,2}$, respectively. The plane containing $E_1$ and $L:=\mathbb{V}(x_0+ax_1,bx_2+cx_3)$ is $H=\mathbb{V}(x_0+ax_1)$. To obtain the third line, say $\Lambda$, contained in $S\cap H$, we factor $f(-ax_1,x_1,x_2,x_3)=x_1(bx_2+cx_3)(mx_1+nx_2+px_3)$. Simplifying, we find the following equations:
\begin{align*}
    bm&=a^2\alpha_{2,0,1,0}-a\alpha_{1,1,1,0}+\alpha_{0,2,1,0},\\
    cm&=a^2\alpha_{2,0,0,1}-a\alpha_{1,1,0,1}+\alpha_{0,2,0,1},\\
    bn&=-a\alpha_{1,0,2,0}+\alpha_{0,1,2,0},\\
    cp&=-a\alpha_{1,0,0,2}+\alpha_{0,1,0,2},\\
    bp+cn&=-a\alpha_{1,0,1,1}+\alpha_{0,1,1,1}.
\end{align*}
Since $L$ is a line, we note that $(b,c)\neq(0,0)$, so $|b|^2+|c|^2>0$. Thus
\begin{align}\label{notn:m}
m=\frac{\bar{b}(a^2\alpha_{2,0,1,0}-a\alpha_{1,1,1,0}+\alpha_{0,2,1,0})+\bar{c}(a^2\alpha_{2,0,0,1}-a\alpha_{1,1,0,1}+\alpha_{0,2,0,1})}{|b|^2+|c|^2}.
\end{align}

Next, since $|b|^4+|c|^2>0$ and $|b|^2+|c|^4>0$, we use the expressions
\begin{align*}
c^2n&=c(bp+cn)-b(cp)\\
&=c(-a\alpha_{1,0,1,1}+\alpha_{0,1,1,1})-b(-a\alpha_{1,0,0,2}+\alpha_{0,1,0,2})
\end{align*}
and
\begin{align*}
b^2p&=b(bp+cn)-c(bn)\\
&=b(-a\alpha_{1,0,1,1}+\alpha_{0,1,1,1})-c(-a\alpha_{1,0,2,0}+\alpha_{0,1,2,0})
\end{align*}
to solve for $n$ and $p$. This yields
\begin{align}\label{notn:n}
n=\frac{\bar{c}^2(c(-a\alpha_{1,0,1,1}+\alpha_{0,1,1,1})-b(-a\alpha_{1,0,0,2}+\alpha_{0,1,0,2}))+\bar{b}(-a\alpha_{1,0,2,0}+\alpha_{0,1,2,0})}{|b|^2+|c|^4}
\end{align}
and
\begin{align}\label{notn:p}
p=\frac{\bar{b}^2(b(-a\alpha_{1,0,1,1}+\alpha_{0,1,1,1})-c(-a\alpha_{1,0,2,0}+\alpha_{0,1,2,0}))+\bar{c}(-a\alpha_{1,0,0,2}+\alpha_{0,1,0,2})}{|b|^4+|c|^2}.
\end{align}

\begin{rem}\label{rem:first}
It follows that the residual line $\Lambda$ in the plane $H$ is given by $\mathbb{V}(x_0+ax_1,mx_1+nx_2+px_3)$, where $m,n,p$ are as above.
\end{rem}

\begin{notn}\label{notn:m_i,n_i,p_i}
Thinking of $m,n,p$ (Equations~\ref{notn:m}, \ref{notn:n}, and \ref{notn:p}) as functions of $a,b,c$, let $(m_1,n_1,p_1)=(m,n,p)(-s_1c_1-t_4,1+s_1c_2,s_1c_3-t_4)$. Likewise, let $(m_2,n_2,p_2)=(m,n,p)(-s_2c_1-t_4,1+s_2c_2,s_2c_3-t_4)$.
\end{notn}

\begin{prop}
We have the equations $L_{1,3}=\mathbb{V}(x_0+(-s_1c_1-t_4)x_1,m_1x_1+n_1x_2+p_1x_3)$ and $C_2=\mathbb{V}(x_0+(-s_2c_1-t_4)x_1,m_2x_1+n_2x_2+p_2x_3)$.
\end{prop}
\begin{proof}
If $(a,b,c)=(-s_1c_1-t_4,1+s_1c_2,s_1c_3-t_4)$, then $L:=\mb{V}(x_0+ax_1,bx_2+cx_3)=C_3$ and hence the residual line is $\Lambda=L_{1,3}$. If $(a,b,c)=(-s_2c_1-t_4,1+s_2c_2,s_2c_3-t_4)$, then $L:=\mb{V}(x_0+ax_1,bx_2+cx_3)=L_{1,2}$ and hence the residual line is $\Lambda=C_2$. Remark~\ref{rem:first} then gives us the desired equations.
\end{proof}

\subsection{$C_1$ and $L_{2,3}$}
We now apply the approach of Section~\ref{sec:c2 and l13} for $L=C_3$ and $L=L_{1,2}$ to solve for $L_{2,3}$ and $C_1$, respectively. We will give a projective change of coordinates $B$ that fixes $E_2$ and takes $L:=\mb{V}(x_0+ax_1,bx_2+cx_3)$ to $BL=\mb{V}(x_0+ax_1,x_2)$. Intersecting the cubic surface $BS=\mb{V}(f\circ B^{-1})$ with the plane $H$ containing $E_2$ and $BL$, we will be able to solve for the third line $\Lambda$ contained in $BS\cap H$. We then obtain the desired line, namely $C_1$ or $L_{2,3}$, as the line $B^{-1}\Lambda$. Let
\begin{align*}
    (B^{-1})^T=\begin{pmatrix}1&0&0&0\\0&1&0&0\\0&0&\frac{\bar{b}}{|b|^2+|c|^2}&\frac{\bar{c}}{|b|^2+|c|^2}\\0&0&c&-b\end{pmatrix}.
\end{align*}

Note that $L$ is a line, so $(b,c)\neq(0,0)$. Since $(b,c)\neq(0,0)$, it follows that $B$ is well-defined and moreover $\det{B}=-1$. We have that $BE_2=E_2$ and $BL=\mb{V}(x_0+ax_1,x_2+(bc-bc)x_3)=\mb{V}(x_0+ax_1,x_2)$. The plane $H=\mb{V}(x_2)$ contains both $E_2$ and $BL$. The intersection $BS\cap H$ is given by the vanishing of
\begin{align*}
    f\circ B^{-1}|_{x_2=0}&=f(x_0,x_1,cx_3,-bx_3)\\
    &=x_3(x_0+ax_1)(h x_0+j x_1+k x_3).
\end{align*}
Evaluating $f\circ B^{-1}|_{x_2=0}$, we obtain the following relations:
\begin{align}
    h&=\alpha_{2,0,1,0}c-\alpha_{2,0,0,1}b,\label{notn:h}\\
    j+ah&=\alpha_{1,1,1,0}c-\alpha_{1,1,0,1}b,\nonumber\\
    k&=\alpha_{1,0,2,0}c^2-\alpha_{1,0,1,1}bc+\alpha_{1,0,0,2}b^2.\label{notn:k}
\end{align}

Subtracting $ah$ from $j+ah$, we have
\begin{align}\label{notn:j}
j=\alpha_{1,1,1,0}c-\alpha_{1,1,0,1}b-a\left(\alpha_{2,0,1,0}c-\alpha_{2,0,0,1}b\right).
\end{align}

\begin{rem}\label{rem:second}
It follows that $BS\cap H$ contains the lines $E_2,BL$, and $\Lambda=\mb{V}(hx_0+jx_1+kx_3,x_2)$. Applying $B^{-1}$, we have
\[B^{-1}\Lambda=\mb{V}(hx_0+jx_1+\tfrac{\bar{c}k}{|b|^2+|c|^2}x_2-\tfrac{\bar{b}k}{|b|^2+|c|^2}x_3,bx_2+cx_3).\]
\end{rem}

\begin{notn}\label{notn:h_i,j_i,k_i}
Thinking of $h,j,k$ (Equations~\ref{notn:h}, \ref{notn:j}, and \ref{notn:k}) as functions of $a,b,c$, let $(h_1,j_1,k_1)=(h,j,k)(-s_1c_1-t_4,1+s_1c_2,s_1c_3-t_4)$. Likewise, let $(h_2,j_2,k_2)=(h,j,k)(-s_2c_1-t_4,1+s_2c_2,s_2c_3-t_4)$.
\end{notn}

\begin{prop}
We have the equations
\begin{align*}
L_{2,3}=\mathbb{V}\big(&h_1x_0+j_1x_1+\tfrac{(\overline{s_1c_3-t_4})k_1}{|1+s_1c_2|^2+|s_1c_3-t_4|^2}x_2-\tfrac{(\overline{1+s_1c_2})k_1}{|1+s_1c_2|^2+|s_1c_3-t_4|^2}x_3,\\
&(1+s_1c_2)x_2+(s_1c_3-t_4)x_3\big)
\end{align*}
and
\begin{align*}
C_1=\mathbb{V}\big(&h_2x_0+j_2x_1+\tfrac{(\overline{s_2c_3-t_4})k_2}{|1+s_2c_2|^2+|s_2c_3-t_4|^2}x_2-\tfrac{(\overline{1+s_2c_2})k_2}{|1+s_2c_2|^2+|s_2c_3-t_4|^2}x_3,\\
&(1+s_2c_2)x_2+(s_2c_3-t_4)x_3\big).
\end{align*}
\end{prop}
\begin{proof}
If $(a,b,c)=(-s_1c_1-t_4,1+s_1c_2,s_1c_3-t_4)$, then $L:=\mb{V}(x_0+ax_1,bx_2+cx_3)=C_3$ and hence the residual line is $\Lambda=BL_{2,3}$. If $(a,b,c)=(-s_2c_1-t_4,1+s_2c_2,s_2c_3-t_4)$, then $L:=\mb{V}(x_0+ax_1,bx_2+cx_3)=L_{1,2}$ and hence the residual line is $\Lambda=BC_1$. Remark~\ref{rem:second} then gives us the desired equations.
\end{proof}

\begin{rem}
If $s_2=\infty$, we can again obtain the correct lines from the above formulas by dividing all terms by $s_2=\infty$ as discussed in Remark~\ref{rem:s_2=infty}.
\end{rem}

\section{The final six lines}\label{sec:step5}
We now want to solve for $E_4,E_5,E_6,L_{4,5},L_{4,6}$, and $L_{5,6}$. For $i,j,k$ distinct elements of $\{4,5,6\}$, we note that $L_{i,j}$ and $E_k$ are the two lines passing through the four skew lines $C_i$, $C_j$, $L_{1,k}$, and $L_{2,k}$. We will use the same methods as in Section~\ref{sec:step3} to solve for these lines. We first give two projective changes of coordinates. Let
\begin{align*}
(A_{i,j}^{-1})^T=\begin{pmatrix}t_j & 1 & 0 & 0\\ t_i & 1 & 0 & 0\\ 0 & 0& t_j & 1\\ 0 & 0 & t_i & 1\end{pmatrix}\quad\text{ and }\quad (B_{i,j,k}^{-1})^T=\begin{pmatrix}1 & 0 & 0 & 0\\ 0 & -\tfrac{t_j-t_k}{t_i-t_k} & 0 & \tfrac{\gamma(t_j-t_k)}{\eps(t_i-t_k)}\\ 0 & 0 &\tfrac{1}{\delta} & 0\\ 0 & 0 & 0 & -\tfrac{1}{\eps}\end{pmatrix},
\end{align*}
where
\begin{align}\label{notn:a,b,c}
\gamma&=(1-\tfrac{t_i-t_k}{t_j-t_k})(t_k^2\alpha_{2,0,1,0}+t_k\alpha_{1,1,1,0}+\alpha_{0,2,1,0}),\\
\delta&=t_k^2\alpha_{1,0,2,0}+t_k(t_j\alpha_{1,0,2,0}+\alpha_{0,1,2,0}+\alpha_{1,0,1,1})+(t_j\alpha_{0,1,2,0}+\alpha_{0,1,1,1}),\nonumber \\
\eps&=t_k^2\alpha_{1,0,2,0}+t_k(t_i\alpha_{1,0,2,0}+\alpha_{0,1,2,0}+\alpha_{1,0,1,1})+(t_i\alpha_{0,1,2,0}+\alpha_{0,1,1,1}).\nonumber
\end{align}

Since $t_i\neq t_j$ (as noted in the proof of Proposition~\ref{prop:c_i}), we have that $A_{i,j}$ is non-singular. As a result, the fact that $C_i,C_j$, and $L_{1,k}$ are skew implies that $A_{i,j}C_i,A_{i,j}C_j$, and $A_{i,j}L_{1,k}$ are skew. Moreover, we have that $A_{i,j}C_i=\mb{V}(x_0,x_2)$ and $A_{i,j}C_j=\mb{V}(x_1,x_3)$. We also have
\begin{align*}
A_{i,j}L_{1,k}&=\mb{V}((t_j-t_k)x_0+(t_i-t_k)x_1,(t_k^2\alpha_{2,0,1,0}+t_k\alpha_{1,1,1,0}+\alpha_{0,2,1,0})(x_0+x_1)+\delta x_2+\eps x_3)\\
&=\mb{V}(x_0+\tfrac{t_i-t_k}{t_j-t_k}x_1,\gamma x_1+\delta x_2+\eps x_3).
\end{align*}

As mentioned in the proof of Proposition~\ref{prop:c_i}, we have that $t_i\neq t_j\neq t_k$, so $\frac{t_i-t_k}{t_j-t_k}$ is a complex number not equal to 0 or 1. Also note that if $\delta=0$, then $AC_j$ and $A_{i,j}L_{1,k}$ intersect at $\pcoor{0:0:1:0}$, contradicting the fact that they are skew. Similarly, if $\eps=0$, then $A_{i,j}C_i$ and $A_{i,j}L_{1,k}$ intersect at $\pcoor{0:0:0:1}$, again contradicting our skew assumption. We thus have that $\delta\neq 0$ and $\eps\neq 0$, so the change of coordinates given by $B_{i,j,k}$ is well-defined and non-singular. Now we have
\begin{align*}
B_{i,j,k}A_{i,j}C_i&=\mb{V}(x_0,x_2),\\
B_{i,j,k}A_{i,j}C_j&=\mb{V}(x_1,x_3),\\
B_{i,j,k}A_{i,j}L_{1,k}&=\mb{V}(x_0-x_1,x_2-x_3).
\end{align*}
These three skew lines lie on the ruling $N_t=\{\pcoor{t:1:bt:b}\}$ of the quadric surface $Q=\mb{V}(x_0x_3-x_1x_2)$. In particular, we have $N_0=B_{i,j,k}A_{i,j}C_i$, $N_\infty=B_{i,j,k}A_{i,j}C_j$, and $N_1=B_{i,j,k}A_{i,j}L_{1,k}$. Next, we will intersect $B_{i,j,k}A_{i,j}L_{2,k}$ with $Q$. By B\'ezout's Theorem, this intersection will consist of two points (which are distinct by the same reasoning outlined in Remark~\ref{rem:distinct}). The lines in the ruling $M_s=\{\pcoor{s:as:1:a}\}$ passing through these two points will be $B_{i,j,k}A_{i,j}L_{i,j}$ and $B_{i,j,k}A_{i,j}E_k$.
We have that
\begin{align*}
B_{i,j,k}A_{i,j}L_{2,k}&=B_{i,j,k}\mb{V}((t_j-t_k)x_2+(t_i-t_k)x_3,\pi x_0+\rho x_1+\sigma x_2+\sigma x_3)\\
&=\mb{V}(\tfrac{\gamma(t_j-t_k)}{\eps}x_1+\tfrac{t_j-t_k}{\delta}x_2-\tfrac{t_i-t_k}{\eps}x_3,\pi x_0+(-\tfrac{\rho(t_j-t_k)}{t_i-t_k}+\tfrac{\sigma \gamma(t_j-t_k)}{\eps(t_i-t_k)})x_1+\tfrac{\sigma}{\delta}x_2-\tfrac{\sigma}{\eps}x_3)\\
&=\mb{V}(\tfrac{\gamma\delta}{\eps}x_1+x_2-\tfrac{\delta(t_i-t_k)}{\eps(t_j-t_k)}x_3,\pi x_0+\tfrac{\sigma\gamma(t_j-t_i)-\rho\eps(t_j-t_k)}{\eps(t_i-t_k)}x_1+\tfrac{\sigma}{\eps}(\tfrac{t_i-t_j}{t_j-t_k})x_3).
\end{align*}
where
\begin{align}\label{notn:pi,rho,sigma}
\pi&=t_k^2\alpha_{2,0,1,0}+t_k(t_j\alpha_{2,0,1,0}+\alpha_{2,0,0,1}+\alpha_{1,1,1,0})+(t_j\alpha_{2,0,0,1}+\alpha_{1,1,0,1}),\\
\rho&=t_k^2\alpha_{2,0,1,0}+t_k(t_i\alpha_{2,0,1,0}+\alpha_{2,0,0,1}+\alpha_{1,1,1,0})+(t_i\alpha_{2,0,0,1}+\alpha_{1,1,0,1}),\nonumber \\
\sigma&=t_k^2\alpha_{1,0,2,0}+t_k\alpha_{1,0,1,1}+\alpha_{1,0,0,2}.\nonumber
\end{align}

\begin{prop}
We have that $\pi\neq 0$.
\end{prop}
\begin{proof}
If $\pi=0$, then $A_{i,j}L_{2,k}=\mb{V}((t_j-t_k)x_2+(t_i-t_k)x_3,\rho x_1+\sigma x_2+\sigma x_3)$. Note that $A_{i,j}C_j=\mb{V}((t_i-t_j)x_1,(t_i-t_j)x_3)=\mb{V}(x_1,x_3)$. Thus the point $\pcoor{1:0:0:0}$ is contained in both $A_{i,j}L_{2,k}$ and $A_{i,j}C_j$, so these lines are not skew. However, this contradicts the fact that $L_{2,k}$ and $C_j$ are skew, so we conclude that $\pi\neq 0$.
\end{proof}

We compute the intersection $Q\cap B_{i,j,k}A_{i,j}L_{2,k}$ by substituting $x_2=-\tfrac{\gamma\delta}{\eps}x_1+\tfrac{\delta(t_i-t_k)}{\eps(t_j-t_k)}x_3$ and $x_0=\tfrac{\sigma\gamma(t_i-t_j)+\rho\eps(t_j-t_k)}{\pi\eps(t_i-t_k)}x_1-\tfrac{\sigma}{\pi\eps}(\tfrac{t_i-t_j}{t_j-t_k})x_3$ into the defining equation for $Q$. We thus have $Q\cap B_{i,j,k}A_{i,j}L_{2,k}=\mb{V}(\tfrac{\gamma\delta}{\eps}x_1^2+(\tfrac{\sigma\gamma(t_i-t_j)+\rho\eps(t_j-t_k)}{\pi\eps(t_i-t_k)}-\tfrac{\delta}{\eps}(\tfrac{t_i-t_k}{t_j-t_k}))x_1x_3-\tfrac{\sigma}{\pi\eps}(\tfrac{t_i-t_j}{t_j-t_k})x_3^2)$. Lines in the ruling $M_s$ are determined by the ratio $\frac{x_1}{x_3}$, so it suffices to solve the quadratic equation $\gamma\delta(\tfrac{x_1}{x_3})^2+(\tfrac{\sigma\gamma(t_i-t_j)+\rho\eps(t_j-t_k)}{\pi(t_i-t_k)}-\delta(\tfrac{t_i-t_k}{t_j-t_k}))\tfrac{x_1}{x_3}-\tfrac{\sigma}{\pi}(\tfrac{t_i-t_j}{t_j-t_k})=0$. These solutions are given by
\begin{align*}
\frac{x_1}{x_3}=\tfrac{1}{2\gamma\delta}\cdot\bigg(&-\tfrac{\sigma\gamma(t_i-t_j)(t_j-t_k)+\rho\eps(t_j-t_k)^2-\delta(t_i-t_k)^2}{\pi(t_i-t_k)(t_j-t_k)}\\
&\pm\sqrt{(\tfrac{\sigma\gamma(t_i-t_j)(t_j-t_k)+\rho\eps(t_j-t_k)^2-\delta(t_i-t_k)^2}{\pi(t_i-t_k)(t_j-t_k)})^2+\tfrac{4\gamma\delta\sigma}{\pi}(\tfrac{t_i-t_j}{t_j-t_k})}\bigg).
\end{align*}

Note that these solutions are given by B\'ezout's Theorem applied to $Q\cap B_{i,j,k}A_{i,j}L_{2,k}$. By Remark~\ref{rem:distinct}, these solutions are necessarily distinct.

\begin{notn}\label{notn:q}
Note that $\gamma,\delta,\eps$ (see Equation~\ref{notn:a,b,c}) and $\pi,\rho,\sigma$ (see Equation~\ref{notn:pi,rho,sigma}) depend on $i,j,k$. Let $\gamma_{i,j,k},\delta_{i,j,k},\eps_{i,j,k},\pi_{i,j,k},\rho_{i,j,k},\sigma_{i,j,k}$ denote the values of $\gamma,\delta,\eps,\pi,\rho,\sigma$ as functions of $i,j,k$. Furthermore, let
\begin{align*}
q_{i,j,k}^+=\tfrac{1}{2\gamma\delta}\cdot\bigg(&-\tfrac{\sigma\gamma(t_i-t_j)(t_j-t_k)+\rho\eps(t_j-t_k)^2-\delta(t_i-t_k)^2}{\pi(t_i-t_k)(t_j-t_k)}\\
&+\sqrt{(\tfrac{\sigma\gamma(t_i-t_j)(t_j-t_k)+\rho\eps(t_j-t_k)^2-\delta(t_i-t_k)^2}{\pi(t_i-t_k)(t_j-t_k)})^2+\tfrac{4\gamma\delta\sigma}{\pi}(\tfrac{t_i-t_j}{t_j-t_k})}\bigg)
\end{align*}
and
\begin{align*}
q_{i,j,k}^-=\tfrac{1}{2\gamma\delta}\cdot\bigg(&-\tfrac{\sigma\gamma(t_i-t_j)(t_j-t_k)+\rho\eps(t_j-t_k)^2-\delta(t_i-t_k)^2}{\pi(t_i-t_k)(t_j-t_k)}\\
&-\sqrt{(\tfrac{\sigma\gamma(t_i-t_j)(t_j-t_k)+\rho\eps(t_j-t_k)^2-\delta(t_i-t_k)^2}{\pi(t_i-t_k)(t_j-t_k)})^2+\tfrac{4\gamma\delta\sigma}{\pi}(\tfrac{t_i-t_j}{t_j-t_k})}\bigg).
\end{align*}
\end{notn}

Remark~\ref{rem:distinct} implies that $q_{i,j,k}^+\neq q_{i,j,k}^-$. It follows that we have the line $M_{q_{i,j,k}^\pm}=\mb{V}(x_0-q_{i,j,k}^\pm x_2,x_1-q_{i,j,k}^\pm x_3)$.

\begin{prop}\label{prop:e_k and l_ij}
We have the equations
\begin{align*}
E_k=\mb{V}(&x_0-t_ix_1-\delta_{i,j,k}q_{i,j,k}^+(x_2-t_ix_3),\\
&(\tfrac{t_i-t_k}{t_j-t_k}-\gamma_{i,j,k}q_{i,j,k}^+)(x_0-t_jx_1)-\eps_{i,j,k}q_{i,j,k}^+(x_2-t_jx_3))
\end{align*}
and
\begin{align*}
L_{i,j}=\mb{V}(&x_0-t_ix_1-\delta_{i,j,k}q_{i,j,k}^-(x_2-t_ix_3),\\
&(\tfrac{t_i-t_k}{t_j-t_k}-\gamma_{i,j,k}q_{i,j,k}^-)(x_0-t_jx_1)-\eps_{i,j,k}q_{i,j,k}^-(x_2-t_jx_3)).
\end{align*}
\end{prop}
\begin{proof}
Without loss of generality, we may assume $B_{i,j,k}A_{i,j}C_k=\mb{V}(x_0-q_{i,j,k}^+x_2,x_1-q_{i,j,k}^+x_3)$ and $B_{i,j,k}A_{i,j}L_{i,j}=\mb{V}(x_0-q_{i,j,k}^-x_2,x_1-q_{i,j,k}^-x_3)$. We thus have
\begin{align*}
C_k&=(B_{i,j,k}A_{i,j})^{-1}\mb{V}(x_0-q_{i,j,k}^+x_2,x_1-q_{i,j,k}^+x_3),\\
L_{i,j}&=(B_{i,j,k}A_{i,j})^{-1}\mb{V}(x_0-q_{i,j,k}^-x_2,x_1-q_{i,j,k}^-x_3).
\end{align*}
The inverse matrices are
\begin{align*}
A_{i,j}^T=\frac{1}{t_j-t_i}\begin{pmatrix}1& -1 & 0 & 0\\ -t_i & t_j & 0 & 0\\ 0 & 0& 1 & -1\\ 0 & 0 & -t_i & t_j\end{pmatrix}\quad\text{ and }\quad B_{i,j,k}^T=\begin{pmatrix}1 & 0 & 0 & 0\\ 0 & -\tfrac{t_i-t_k}{t_j-t_k} & 0 & -\gamma_{i,j,k}\\ 0 & 0 &\delta_{i,j,k} & 0\\ 0 & 0 & 0 & -\eps_{i,j,k}\end{pmatrix}.
\end{align*}
\end{proof}

\section{The general case}\label{sec:general}
Let $S'=\mb{V}(f')$ be a smooth cubic surface, where 
\[f'(x_0,x_1,x_2,x_3)=\sum_{i+j+k+l=3}\beta_{i,j,k,l}x_0^ix_1^jx_2^kx_3^l.\]
Moreover, let
\begin{align*}
\Lambda_1&=\mb{V}\big(\sum_{i=0}^3\mf{a}_ix_i,\sum_{i=0}^3\mf{a}_i'x_i\big),\\
\Lambda_2&=\mb{V}\big(\sum_{i=0}^3\mf{b}_ix_i,\sum_{i=0}^3\mf{b}_i'x_i\big),\\
\Lambda_3&=\mb{V}\big(\sum_{i=0}^3\mf{c}_ix_i,\sum_{i=0}^3\mf{c}_i'x_i\big)
\end{align*}
be three skew lines contained in $S'$. We will give a projective change of coordinates $A$ taking $\Lambda_i$ to $E_i$ for $1\leq i\leq 3$. Applying the work of the previous sections of the paper, we will have formulas for all 27 lines on $AS'$, with each $\alpha_{i,j,k,l}$ being given by a formula in terms of the $\beta_{i,j,k,l}$. The formulas for the 27 lines on $S'$ will then be obtained by applying $A^{-1}$. Consider the matrix
\[(B^{-1})^T=\begin{pmatrix}
\mf{a}_0&\mf{a}_0'&\mf{b}_0&\mf{b}_0'\\
\mf{a}_1&\mf{a}_1'&\mf{b}_1&\mf{b}_1'\\
\mf{a}_2&\mf{a}_2'&\mf{b}_2&\mf{b}_2'\\
\mf{a}_3&\mf{a}_3'&\mf{b}_3&\mf{b}_3'
\end{pmatrix},\]

which gives $BE_1=\Lambda_1$ and $BE_2=\Lambda_2$. Since $\Lambda_1$ and $\Lambda_2$ are skew, $B$ is non-singular. Next, we will give a projective change of coordinates $C$ that fixes $E_1$ and $E_2$ and takes $B^{-1}\Lambda_3$ to $E_3$. The composite change of coordinates $CB^{-1}$ will then be the desired change of coordinates $A$. Let
\[B^{-1}\Lambda_3=\mb{V}\big(\sum_{i=0}^3\mf{c}_ix_i,\sum_{i=0}^3\mf{c}_i'x_i\big).\]

Since $B$ is non-singular, the lines $E_1,E_2$, and $B^{-1}\Lambda_3$ are skew. Thus $B^{-1}\Lambda_3$ is not a subspace of $\{x_0=0\}$ or $\{x_3=0\}$, so $B^{-1}\Lambda_3$ is determined by the points $B^{-1}\Lambda_3\cap\{x_0=0\}=\pcoor{0:a:b:c}$ and $B^{-1}\Lambda_3\cap\{x_3=0\}=\pcoor{d:e:f:0}$. Moreover, since $B^{-1}\Lambda_3$ does not meet $E_1$ or $E_2$, we may assume that $B^{-1}\Lambda_3\cap\{x_0=0\}=\pcoor{0:1:b:c}$ and $B^{-1}\Lambda_3\cap\{x_3=0\}=\pcoor{d:e:1:0}$. In terms of the defining equations for $B^{-1}\Lambda_3$, we have
\begin{align*}
b&=\frac{\mf{c}_1'\mf{c}_3-\mf{c}_1\mf{c}_3'}{\mf{c}_2\mf{c}_3'-\mf{c}_2'\mf{c}_3},\quad c=\frac{\mf{c}_1\mf{c}_2'-\mf{c}_1'\mf{c}_2}{\mf{c}_2\mf{c}_3'-\mf{c}_2'\mf{c}_3},\\
d&=\frac{\mf{c}_1\mf{c}_2'-\mf{c}_1'\mf{c}_2}{\mf{c}_0\mf{c}_1'-\mf{c}_0'\mf{c}_1},\quad e=\frac{\mf{c}_0'\mf{c}_2-\mf{c}_0\mf{c}_2'}{\mf{c}_0\mf{c}_1'-\mf{c}_0'\mf{c}_1}.
\end{align*}

Note that $c$ and $d$ are either both zero or both non-zero. If $c,d$ are both zero, then we instead construct a projective change of coordinates taking $B^{-1}\Lambda_3\cap\{x_1=0\}$ and $B^{-1}\Lambda_3\cap\{x_2=0\}$ to $\pcoor{1:0:1:0}$ and $\pcoor{0:1:0:1}$, respectively. We omit these calculations and simply discuss the case when $c,d$ are non-zero. If $c,d$ are non-zero, the projective change of coordinates given by
\[C=\begin{pmatrix}\tfrac{1}{d}&0&0&0\\-\tfrac{e}{d}&1&0&0\\0&0&1&-\tfrac{b}{c}\\0&0&0&\tfrac{1}{c}\end{pmatrix}\]
gives us $C(\pcoor{0:1:b:c})=\pcoor{0:1:0:1}$ and $C(\pcoor{d:e:1:0})=\pcoor{1:0:1:0}$. Thus $CB^{-1}\Lambda_3=E_3$. Moreover, $CE_1=E_1$ and $CE_2=E_2$, so the projective change of coordinates $A=CB^{-1}$ takes $\Lambda_1,\Lambda_2,\Lambda_3$ to $E_1,E_2,E_3$, as desired. We may thus apply the work done in previous sections to the surface $CB^{-1}S'$, where the $\alpha_{i,j,k,l}$ will now be determined as functions of $\beta_{i,j,k,l}$. For each line $L\subset S$, we then get a line $BC^{-1}L\subset S'$.

\section{Smooth cubic surfaces over $\mb{R}$}\label{sec:real case}
Over the real numbers, Schl\"afli showed that a smooth cubic surface contains 3, 7, 15, or 27 lines~\cite{Sch58}. Segre further classifies these lines into two types, namely {\it hyperbolic lines} and {\it elliptic lines}~\cite{Seg42}. Finashin--Kharlamov~\cite{FK13} and Okonek--Teleman~\cite{OT11} note that Segre in fact proved that the difference between the number $h$ of hyperbolic lines and the number $e$ of elliptic lines on a real smooth cubic surface is always 3. We note that if we are given three skew lines on a real smooth cubic surface $S$, then we have at least one real root of $g(t)$ (see Proposition~\ref{prop:c_i}). Without loss of generality, we may assume that $t_4$ is a real root of $g(t)$, and we thus have that the line $C_4$ is defined over $\mb{R}$. In this case, $S$ contains more than three lines and therefore must contain elliptic lines. As a result, we have proved the following proposition.

\begin{prop}\label{prop:real-only3}
If $S$ is a real smooth cubic surface that contains no elliptic lines, then the three lines contained in $S$ are not skew.
\end{prop}

In fact, we can prove that $S$ contains three skew lines if and only if $S$ contains an elliptic line. First, we prove a basic graph theoretic fact that will simplify our argument.

\begin{prop}\label{prop:graph}
Let $G$ be a graph of order at least seven, such that for any triple of vertices $v_1,v_2,v_3$, at least two of $v_1,v_2,v_3$ are connected by an edge. Then $G$ contains two distinct 3-cycles that share an edge.
\end{prop}
\begin{proof}
If $G$ has at least three connected components, then three vertices coming from distinct components do not share any edges, so $G$ can have at most two connected components. If $G$ has two connected components (say $G_1$ and $G_2$), then one component of $G$ has at least four vertices. Without loss of generality, we may assume that $G_1$ has at least four vertices. Taking a vertex from $G_2$, the component $G_1$ must have diameter 1, which implies that $G_1$ contains two distinct 3-cycles that share an edge. 

Finally, suppose that $G$ is connected. Fixing a vertex $v$ of $G$, the subgraph $G'$ of vertices that are distance greater than 1 from $v$ must have diameter 1. If $G'$ has four or more vertices, then $G$ contains two distinct 3-cycles that share an edge. If $G'$ contains zero or one vertex, then $v$ has at least five adjacent vertices. Any triple of these $v$-adjacent vertices must have at least one edge between them, which forces $G$ to contain two distinct 3-cycles that share an edge. If $G'$ contains two vertices, then $G$ contains the graph illustrated in Figure~\ref{fig:graph1}. If $G'$ contains three vertices, then $G$ contains the graph illustrated in Figure~\ref{fig:graph2}. In either case, we select three vertices that are pairwise non-adjacent and add an edge between two of them. Repeating this process will always yield two distinct 3-cycles that share an edge, as desired.
\end{proof}

\begin{figure}[h]
\begin{minipage}[b]{0.45\linewidth}
\centering
\begin{tikzpicture}
\node at (0,-0.5) {$v$};
\draw[fill=black] (0,0) circle (3pt);
\draw[fill=black] (1,0) circle (3pt);
\draw[fill=black] (-1,1) circle (3pt);
\draw[fill=black] (0,1) circle (3pt);
\draw[fill=black] (-1,0) circle (3pt);
\draw[fill=black] (1,1) circle (3pt);
\draw[fill=black] (2,0) circle (3pt);
\draw[thick] (-1,0) -- (0,0) -- (1,1) -- (1,0) -- (2,0);
\draw[thick] (-1,1) -- (0,0);
\draw[thick] (0,0) -- (0,1);
\end{tikzpicture}
\caption{}\label{fig:graph1}
\end{minipage}
\hspace{0.5cm}
\begin{minipage}[b]{0.45\linewidth}
\centering
\begin{tikzpicture}
\node at (0,-0.5) {$v$};
\draw[fill=black] (0,0) circle (3pt);
\draw[fill=black] (1,1) circle (3pt);
\draw[fill=black] (1,0) circle (3pt);
\draw[fill=black] (0,1) circle (3pt);
\draw[fill=black] (2,0) circle (3pt);
\draw[fill=black] (3,1) circle (3pt);
\draw[fill=black] (2,1) circle (3pt);
\draw[thick] (0,0) -- (1,1);
\draw[thick] (0,0) -- (1,0);
\draw[thick] (0,0) -- (0,1);
\draw[thick] (1,0) -- (2,0) -- (2,1) -- (3,1) -- (2,0);
\end{tikzpicture}
\caption{}\label{fig:graph2}
\end{minipage}
\end{figure}

\begin{lem}\label{lem:real contains skew}
A real smooth cubic surface $S$ contains three skew lines if and only if $S$ contains an elliptic line.
\end{lem}
\begin{proof}
By Proposition~\ref{prop:real-only3} and Schl\"afli's count of lines on a real smooth cubic surface, we may assume that $S$ contains at least seven real lines, say $\Lambda_1,...,\Lambda_7$. We represent $\{\Lambda_i\}$ and their intersections as a graph $G$. The vertices of $G$ are given by the lines $\Lambda_i$, and vertices are connected by an edge whenever the corresponding lines intersect each other. Note that a 3-cycle corresponds to three coplanar lines. By B\'ezout's Theorem, the plane containing these lines cannot intersect $S$ in another line, so we cannot have two distinct 3-cycles in $G$ that share an edge. The contrapositive of Proposition~\ref{prop:graph} implies that $G$ has three vertices with no shared edge among them, which means that $S$ contains three skew lines.
\end{proof}

We are now prepared to give a proof of Theorem~\ref{thm:main theorem}. If $S$ is a real smooth cubic surface that contains an elliptic line, then we can determine the number of real lines contained in $S$ by analyzing the formulas obtained in this paper.

\begin{proof}[Proof of Theorem~\ref{thm:main theorem}]
By Lemma~\ref{lem:real contains skew}, $S$ contains three skew lines. Without loss of generality, we may assume that $S$ contains the lines $E_1=\mb{V}(x_0,x_1)$, $E_2=\mb{V}(x_2,x_3)$, $E_3=\mb{V}(x_0-x_2,x_1-x_3)$ and that $t_4$ is a real root of $g(t)$. We thus have that the lines $C_4,L_{1,4},L_{2,4},L_{3,4}$ are defined over $\mb{R}$. 

If $S$ contains exactly 7 real lines, then this accounts for all lines contained in $S$, so $g(t)$ can only have one real root. Moreover, Proposition~\ref{prop:c_3 and l_12} implies that $C_3$ and $L_{1,2}$ are defined over $\mb{R}(t_4,s_1)$ and $\mb{R}(t_4,s_2)$ respectively, so $s_1$ and $s_2$ cannot be real numbers in this case. Conversely, if $g(t)$ only has one real root, then $C_5,C_6,L_{i,5},L_{i,6}$ are not real for $1\leq i\leq 3$, so $S$ contains at most 19 lines. Furthermore, if $s_1,s_2\not\in\mb{R}$, then $C_3$ and $L_{1,2}$ are not defined over $\mb{R}$. If two coplanar lines are real, their residual line must also be real. It follows that $C_1,C_2,L_{1,3}$ are not defined over $\mb{R}$, as these are coplanar with $E_2,E_1,E_1$ and residual to $L_{1,2},L_{1,2},C_3$ respectively. Thus $S$ contains at most 14 real lines, so $S$ must contain exactly 7 lines. This proves (a). 

If all roots of $g(t)$ and $s_1,s_2$ are real, then all lines computed in Sections~\ref{sec:step1}--\ref{sec:step4} are real. Moreover, Harris shows that the remaining lines on $S$ are rationally determined~\cite[p. 719]{Har79}, which gives us that all lines on $S$ are real. Conversely, if a root of $g(t)$ or $s_1,s_2$ were not real, then some of the lines in $S$ would not be defined over $\mb{R}$, proving (c). 

Finally, if all roots of $g(t)$ are real and $s_1,s_2$ are not real, then our process gives us all the lines up until $C_3$ and $L_{1,2}$ (see Sections~\ref{sec:step1} and \ref{sec:step2}), yielding a total of 15 lines on $S$. Moreover, the lines $C_3$ and $L_{1,2}$ are not real by Proposition~\ref{prop:c_3 and l_12}, so $S$ contains fewer than 27 real lines and hence contains exactly 15 real lines. Similarly, if $g(t)$ has only one real root (which we label $t_4$) and $s_1,s_2$ are real, then precisely the lines $L_{5,6},E_i,C_j,L_{i,j}$ are real for $1\leq i,j\leq 4$. Conversely, suppose $S$ contains exactly 15 real lines. Then part (a) and part (c) imply that either $g(t)$ has one real root and $s_1,s_2$ are real, or all roots of $g(t)$ are real and $s_1,s_2$ are not real, which proves (b). 
\end{proof}

\bibliography{lines-cubic-surfaces}{}

\begin{thebibliography}{McK21}

\bibitem[EH16]{EH16}
David Eisenbud and Joe Harris.
\newblock {\em 3264 and All That: A Second Course in Algebraic Geometry}.
\newblock Cambridge University Press, 2016.

\bibitem[FK12]{FK13}
Sergey Finashin and Viatcheslav Kharlamov.
\newblock Abundance of real lines on real projective hypersurfaces.
\newblock {\em International Mathematics Research Notices},
  2013(16):3639--3646, Jun 2012.

\bibitem[Har79]{Har79}
Joe Harris.
\newblock Galois groups of enumerative problems.
\newblock {\em Duke Math. J.}, 46(4):685--724, 12 1979.

\bibitem[Jor57]{Jor57}
Camille Jordan.
\newblock {\em Trait\'{e} des substitutions et des \'{e}quations
  alg\'{e}briques}.
\newblock Librairie Scientifique et Technique A. Blanchard, Paris, 1957.
\newblock Nouveau tirage.

\bibitem[McK21]{McK21}
Stephen McKean.
\newblock Rational lines on smooth cubic surfaces, 2021.

\bibitem[OT14]{OT11}
Christian Okonek and Andrei Teleman.
\newblock Intrinsic signs and lower bounds in real algebraic geometry.
\newblock {\em J. Reine Angew. Math.}, 688:219--241, 2014.

\bibitem[PSS19]{PSS19}
Marta Panizzut, Emre~Can Sert{\"o}z, and Bernd Sturmfels.
\newblock An octanomial model for cubic surfaces, 2019.

\bibitem[Sch53]{Sch58}
Ludwig Schl\"{a}fli.
\newblock {\em Gesammelte mathematische {A}bhandlungen. {B}and {II}}.
\newblock Verlag Birkh\"{a}user, Basel, 1953.

\bibitem[Seg42]{Seg42}
B.~Segre.
\newblock {\em The {N}on-singular {C}ubic {S}urfaces}.
\newblock Oxford University Press, Oxford, 1942.

\end{thebibliography}
\bibliographystyle{alpha}
\addresseshere

\newpage
\appendix
\section{Visualizations of real cubic surfaces}\label{sec:graphics}
Using the formulas generated in this paper, we are able to write down explicit equations for real cubic surfaces with 27, 15, or 7 lines. Let
\begin{align*}
f_1=&\ x_0^2x_2-x_0x_2^2+x_0^2x_3-x_0x_1x_2+\tfrac{17}{39}x_1x_2^2-\tfrac{17}{39}x_0x_2x_3\\
&+2x_1^2x_2-3x_0x_1x_3+\tfrac{12}{13}x_0x_3^2+\tfrac{1}{13}x_1x_2x_3,\\
f_2=&\ x_0^2x_2-x_0x_2^2+x_0^2x_3-x_0x_1x_2+x_1^2x_2-2x_0x_1x_3+x_1x_2^2\\
&-x_0x_2x_3-x_0x_3^2+2x_1x_2x_3,\\
f_3=&\ x_0^2x_2-x_0x_2^2+2x_0^2x_3-2x_0x_1x_2+x_1^2x_2-x_0x_1x_3+x_1^2x_3-x_1x_3^2.
\end{align*}

Figure~\ref{fig:27lines} shows the vanishing of $f_1$ as a real cubic surface with its 27 lines. Figure~\ref{fig:15lines} shows the vanishing of $f_2$ as a real cubic surface with its 15 lines. Figure~\ref{fig:7lines} shows the vanishing of $f_3$ as a real cubic surface with its 7 lines. These figures were generated by Steve Trettel using the equations above.

\begin{figure}[h]
\includegraphics[trim = 2 2 2 2,clip]{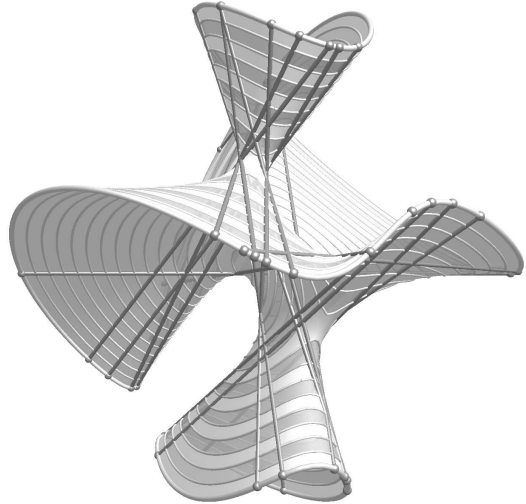}
\caption{Real cubic surface with 27 lines}\label{fig:27lines}
\end{figure}

\begin{figure}[h]
\includegraphics[trim = 2 2 2 2,clip]{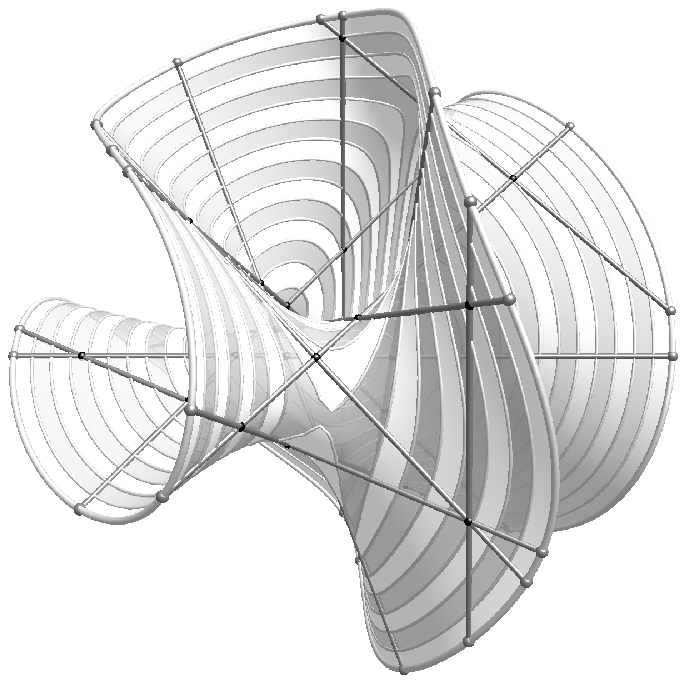}
\caption{Real cubic surface with 15 lines}\label{fig:15lines}
\end{figure}

\begin{figure}[h]
\includegraphics[trim = 2 2 2 2,clip]{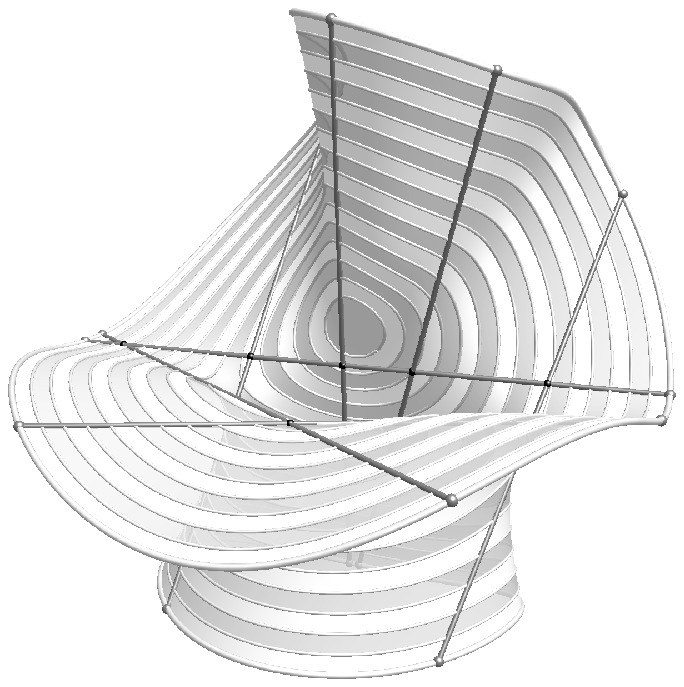}
\caption{Real cubic surface with 7 lines}\label{fig:7lines}
\end{figure}

\newpage
\section{Table of lines}\label{sec:table}
In the following tables, we describe a line $L=\mb{V}\displaystyle\big(\sum_{i=0}^3 a_ix_i,\sum_{i=0}^3 b_ix_i\big)$ by listing its coefficients $a_0,...,a_3,b_0,...,b_3$ as follows:

\begin{table}[h]
\renewcommand{\arraystretch}{1.6}
\aboverulesep=0ex
\belowrulesep=0ex
\centering
\begin{tabular}{@{}l|c|c|c|c@{}}
\toprule
$L$ & $a_0$ & $a_1$ & $a_2$ & $a_3$ \\& $b_0$ & $b_1$ & $b_2$ & $b_3$\\
\bottomrule
\end{tabular}
\end{table}
We also provide references to the relevant notation from throughout the paper.
\begin{table}[h]
\centering
\begin{tabular}{@{}cl@{}}
\toprule
$\alpha_{i,j,k,l}$ & Section~\ref{sec:surface}\\
\midrule
$t_4,t_5,t_6$ & Proposition~\ref{prop:c_i}\\
\midrule
$c_1,c_2,c_3,d_1,d_2,d_3$ & Notation~\ref{notn:c_i and d_i}\\
\midrule
$u_1,u_2,u_3,v_2,v_3$ & Notation~\ref{notn:u_i and v_i}\\
$s_1,s_2$ & Notation~\ref{notn:s_i}\\
\midrule
$m_1,n_1,p_1,m_2,n_2,p_2$ & Notation~\ref{notn:m_i,n_i,p_i}\\
\midrule
$h_1,j_1,k_1,h_2,j_2,k_2$ & Notation~\ref{notn:h_i,j_i,k_i}\\
\midrule
$\gamma_{i,j,k},\delta_{i,j,k},\eps_{i,j,k}$ & Equation~\ref{notn:a,b,c}\\
$\pi_{i,j,k},\rho_{i,j,k},\sigma_{i,j,k}$ & Equation~\ref{notn:pi,rho,sigma}\\
$q_{i,j,k}^\pm$ & Notation~\ref{notn:q}\\
\bottomrule
\end{tabular}
\end{table}

\begin{table}[h]
\begin{adjustbox}{totalheight=\textheight}
\renewcommand{\arraystretch}{1.6}
\aboverulesep=0ex
 \belowrulesep=0ex
\centering
\begin{tabular}{@{}l|c|c|c|c@{}}
\toprule
$E_1$ & \phantom{filler}1\phantom{filler} & 0 & 0 & 0 \\& 0 & 1 & 0 & 0\\ \midrule
$E_2$ & 0 & 0 & 1 & 0 \\& 0 & 0 & 0 & 1\\ \midrule
$E_3$ & 1 & 0 & $-1$ & 0\\ & 0 & 1 & 0 & $-1$\\ \midrule
$E_4$ & $1$& $-t_5$& $-\delta_{5,6,4}q_{5,6,4}^+$ & $t_5\delta_{5,6,4}q_{5,6,4}^+$\\ & $\tfrac{t_5-t_4}{t_6-t_4}-\gamma_{5,6,4}q_{5,6,4}^+$& $-t_6(\tfrac{t_5-t_4}{t_6-t_4}-\gamma_{5,6,4}q_{5,6,4}^+)$& $-\eps_{5,6,4}q_{5,6,4}^+$ & $t_6\eps_{5,6,4}q_{5,6,4}^+$\\ \midrule
$E_5$ & $1$& $-t_4$& $-\delta_{4,6,5}q_{4,6,5}^+$ & $t_4\delta_{4,6,5}q_{4,6,5}^+$\\ & $\tfrac{t_4-t_5}{t_6-t_5}-\gamma_{4,6,5}q_{4,6,5}^+$& $-t_6(\tfrac{t_4-t_5}{t_6-t_5}-\gamma_{4,6,5}q_{4,6,5}^+)$& $-\eps_{4,6,5}q_{4,6,5}^+$ & $t_6\eps_{4,6,5}q_{4,6,5}^+$\\ \midrule
$E_6$ & $1$& $-t_4$& $-\delta_{4,5,6}q_{4,5,6}^+$ & $t_4\delta_{4,5,6}q_{4,5,6}^+$\\ & $\tfrac{t_4-t_6}{t_5-t_6}-\gamma_{4,5,6}q_{4,5,6}^+$& $-t_5(\tfrac{t_4-t_6}{t_5-t_6}-\gamma_{4,5,6}q_{4,5,6}^+)$& $-\eps_{4,5,6}q_{4,5,6}^+$ & $t_5\eps_{4,5,6}q_{4,5,6}^+$\\ \midrule
$C_1$ & $h_2$ &$j_2$ &$\tfrac{(\overline{s_2c_3-t_4})k_2}{|1+s_2c_2|^2+|s_2c_3-t_4|^2}$ & $-\tfrac{(\overline{1+s_2c_2})k_2}{|1+s_2c_2|^2+|s_2c_3-t_4|^2}$\\ &0 &0 &$1+s_2c_2$ &$s_2c_3-t_4$\\ \midrule
$C_2$ & 1& $-s_2c_1-t_4$&0 &0\\ &0 &$m_2$ &$n_2$ &$p_2$\\ \midrule
$C_3$ & 1 & $-s_1c_1-t_4$ &0 &0\\ & 0&0 & $1+s_1c_2$ & $s_1c_3-t_4$\\ \midrule
$C_4$ & \phantom{filler}1\phantom{filler}& $-t_4$&0 &0 \\&0 &0 &1 &$-t_4$\\ \midrule
$C_5$ & 1& $-t_5$&0 &0 \\&0 &0 &1 &$-t_5$\\ \midrule
$C_6$ & 1& $-t_6$&0 &0 \\&0 &0 &1 &$-t_6$\\ \midrule
$L_{1,2}$ & 1 & $-s_2c_1-t_4$ &0 &0\\ & 0&0 & $1+s_2c_2$ & $s_2c_3-t_4$\\ \midrule
$L_{1,3}$ & 1& $-s_1c_1-t_4$&0 &0\\ &0 &$m_1$ &$n_1$ &$p_1$\\ \midrule
$L_{1,4}$ & 1& $-t_4$& 0&0\\ &0&$t_4^2\alpha_{2,0,1,0}+t_4\alpha_{1,1,1,0}+\alpha_{0,2,1,0}$ &$t_4\alpha_{1,0,2,0}+\alpha_{0,1,2,0}$ &$t_4^2\alpha_{1,0,2,0}+t_4(\alpha_{0,1,2,0}+\alpha_{1,0,1,1})+\alpha_{0,1,1,1}$ \\ \midrule
$L_{1,5}$ & 1& $-t_5$& 0&0\\ &0&$t_5^2\alpha_{2,0,1,0}+t_5\alpha_{1,1,1,0}+\alpha_{0,2,1,0}$ &$t_5\alpha_{1,0,2,0}+\alpha_{0,1,2,0}$ &$t_5^2\alpha_{1,0,2,0}+t_5(\alpha_{0,1,2,0}+\alpha_{1,0,1,1})+\alpha_{0,1,1,1}$ \\ \midrule
$L_{1,6}$ & 1& $-t_6$& 0&0\\ &0&$t_6^2\alpha_{2,0,1,0}+t_6\alpha_{1,1,1,0}+\alpha_{0,2,1,0}$ &$t_6\alpha_{1,0,2,0}+\alpha_{0,1,2,0}$ &$t_6^2\alpha_{1,0,2,0}+t_6(\alpha_{0,1,2,0}+\alpha_{1,0,1,1})+\alpha_{0,1,1,1}$ \\ \midrule
$L_{2,3}$ & $h_1$ &$j_1$ &$\tfrac{(\overline{s_1c_3-t_4})k_1}{|1+s_1c_2|^2+|s_1c_3-t_4|^2}$ & $-\tfrac{(\overline{1+s_1c_2})k_1}{|1+s_1c_2|^2+|s_1c_3-t_4|^2}$\\ &0 &0 &$1+s_1c_2$ &$s_1c_3-t_4$\\ \midrule
$L_{2,4}$ & 0& 0& 1&$-t_4$\\ &$t_4\alpha_{2,0,1,0}+\alpha_{2,0,0,1}$ &$t_4^2\alpha_{2,0,1,0}+t_4(\alpha_{2,0,0,1}+\alpha_{1,1,1,0})+\alpha_{1,1,0,1}$ &0 &$t_4^2\alpha_{1,0,2,0}+t_4\alpha_{1,0,1,1}+\alpha_{1,0,0,2}$\\ \midrule
$L_{2,5}$ & 0& 0& 1&$-t_5$\\&$t_5\alpha_{2,0,1,0}+\alpha_{2,0,0,1}$ &$t_5^2\alpha_{2,0,1,0}+t_5(\alpha_{2,0,0,1}+\alpha_{1,1,1,0})+\alpha_{1,1,0,1}$ & 0&$t_5^2\alpha_{1,0,2,0}+t_5\alpha_{1,0,1,1}+\alpha_{1,0,0,2}$ \\ \midrule
$L_{2,6}$ & 0& 0& 1&$-t_6$\\&$t_6\alpha_{2,0,1,0}+\alpha_{2,0,0,1}$ &$t_6^2\alpha_{2,0,1,0}+t_6(\alpha_{2,0,0,1}+\alpha_{1,1,1,0})+\alpha_{1,1,0,1}$ & 0&$t_6^2\alpha_{1,0,2,0}+t_6\alpha_{1,0,1,1}+\alpha_{1,0,0,2}$ \\ \midrule
$L_{3,4}$ & 1& $-t_4$& $-1$&$t_4$\\ &0 &$t_4^2\alpha_{2,0,1,0}+t_4\alpha_{1,1,1,0}+\alpha_{0,2,1,0}$ &$t_4\alpha_{2,0,1,0}+\alpha_{0,1,2,0}+\alpha_{1,1,1,0}$ &$t_4\alpha_{2,0,0,1}-\alpha_{1,0,0,2}$\\ \midrule
$L_{3,5}$ & 1& $-t_5$& $-1$&$t_5$\\ &0 &$t_5^2\alpha_{2,0,1,0}+t_5\alpha_{1,1,1,0}+\alpha_{0,2,1,0}$ &$t_5\alpha_{2,0,1,0}+\alpha_{0,1,2,0}+\alpha_{1,1,1,0}$ &$t_5\alpha_{2,0,0,1}-\alpha_{1,0,0,2}$\\ \midrule
$L_{3,6}$ & 1& $-t_6$& $-1$&$t_6$\\ &0 &$t_6^2\alpha_{2,0,1,0}+t_6\alpha_{1,1,1,0}+\alpha_{0,2,1,0}$ &$t_6\alpha_{2,0,1,0}+\alpha_{0,1,2,0}+\alpha_{1,1,1,0}$ &$t_6\alpha_{2,0,0,1}-\alpha_{1,0,0,2}$\\ \midrule
$L_{4,5}$ & $1$& $-t_4$& $-\delta_{4,5,6}q_{4,5,6}^-$ & $t_4\delta_{4,5,6}q_{4,5,6}^-$\\ & $\tfrac{t_4-t_6}{t_5-t_6}-\gamma_{4,5,6}q_{4,5,6}^-$& $-t_5(\tfrac{t_4-t_6}{t_5-t_6}-\gamma_{4,5,6}q_{4,5,6}^-)$& $-\eps_{4,5,6}q_{4,5,6}^-$ & $t_5\eps_{4,5,6}q_{4,5,6}^-$\\ \midrule
$L_{4,6}$ & $1$& $-t_4$& $-\delta_{4,6,5}q_{4,6,5}^-$ & $t_4\delta_{4,6,5}q_{4,6,5}^-$\\ & $\tfrac{t_4-t_5}{t_6-t_5}-\gamma_{4,6,5}q_{4,6,5}^-$& $-t_6(\tfrac{t_4-t_5}{t_6-t_5}-\gamma_{4,6,5}q_{4,6,5}^-)$& $-\eps_{4,6,5}q_{4,6,5}^-$ & $t_6\eps_{4,6,5}q_{4,6,5}^-$\\ \midrule
$L_{5,6}$ & $1$& $-t_5$& $-\delta_{5,6,4}q_{5,6,4}^-$ & $t_5\delta_{5,6,4}q_{5,6,4}^-$\\ & $\tfrac{t_5-t_4}{t_6-t_4}-\gamma_{5,6,4}q_{5,6,4}^-$& $-t_6(\tfrac{t_5-t_4}{t_6-t_4}-\gamma_{5,6,4}q_{5,6,4}^-)$& $-\eps_{5,6,4}q_{5,6,4}^-$ & $t_6\eps_{5,6,4}q_{5,6,4}^-$\\
\bottomrule
\end{tabular}
\end{adjustbox}
\end{table}

\end{document}